\documentclass{amsart} 
\usepackage{amssymb}
\usepackage{amsfonts}
\usepackage{amsmath}
\usepackage{graphicx}
\usepackage[english]{babel}

\usepackage{hyperref} 

\usepackage{times}

\numberwithin{equation}{section}

\newtheorem{prop}{Proposition}[section]
\newtheorem{prob}[prop]{Problem}
\newtheorem{thm}[prop]{Theorem}
\newtheorem{ex}[prop]{Example}

\newtheorem{rmk}[prop]{Remark}
\newtheorem{lem}[prop]{Lemma}

\newtheorem{coro}[prop]{Corollary}

\allowdisplaybreaks[4] 

\begin{document}

\title{Large deviation principle of Multiplicative Ising models on Markov-Cayley Trees}

\author[Jung-Chao Ban]{Jung-Chao Ban}
\address[Jung-Chao Ban]{Department of Mathematical Sciences, National Chengchi University, Taipei 11605, Taiwan, ROC.}
\address{Math. Division, National Center for Theoretical Science, National Taiwan University, Taipei 10617, Taiwan. ROC.}
\email{jcban@nccu.edu.tw}

\author[Wen-Guei Hu]{Wen-Guei Hu}
\address[Wen-Guei Hu]{College of Mathematics, Sichuan University, Chengdu, 610065, China}
\email{wghu@scu.edu.cn}

\author[Zongfan Zhang]{Zongfan Zhang}
\address[Zongfan Zhang]{College of Mathematics, Sichuan University, Chengdu, 610065, China}
\email{zfzhang@stu.scu.edu.cn}


\thanks{Ban is partially supported by the National Science and Technology Council, ROC (Contract NSTC 111-2115-M-004-005-MY3). Hu is partially supported by the National Natural Science Foundation of China (Grant No.12271381).}

 \baselineskip=1.2\baselineskip 

\begin{abstract}
  In this paper, we study the large deviation principle (LDP) for two types (Type I and Type II) of multiplicative Ising models. For Types I and II, the explicit formulas for the free energy functions and the associated rate functions are derived. Furthermore, we prove that those free energy functions are differentiable, which indicates that both systems are characterized by a lack of phase transition phenomena.
\end{abstract}

\maketitle


\section{Introduction}

In this article, we study the large deviation principle (LDP) for the
lattice spin systems with Ising $\pm 1$ spins on a Markov-Cayley tree $\mathcal{T}$ (defined later). Before presenting the main findings, below is an explanation of the motivation behind this study. The \emph{large deviation principle}, broadly speaking, is premised on the understanding that the asymptotic behavior of probabilities  $\mathbb{P} \{\frac{1}{N}S_{N}\in \Gamma \}$ as $N\rightarrow \infty$,  where  $\Gamma \subseteq \mathbb{R}$ and $\frac{1}{N}S_{N}$ is an ergodic average. The LDP problem for the standard average $\frac{1}{N}S_{N}=\frac{1}{N}\sum_{i=1}^{N}X_{i}$,  where $\{X_{i}\}_{i=1}^{\infty }$ is a sequence of  i.i.d. random variables or driven by a dynamical system, has been thoroughly investigated,  with the best general references being  \cite{dembo2009ldp, ellis2006entropy}.

Motivated by the multiple ergodic theory (\cite{assani2014survey, furstenberg2014recurrence, furstenberg1982ergodic, furstenberg2011mean})  or nonconventional ergodic theory\footnote{The term `nonconventional' was introduced in \cite{furstenberg1988nonconventional}, and now is standard in ergodic theory.}
(cf.\cite{furstenberg1988nonconventional, kifer2010nonconventional, kifer2013strong, kifer2014nonconventional, krause2022pointwise}), the LDP for the multiple ergodic average can be framed in the following terms.  Let $(X,\mathcal{B},\mu )$ be a probability space, $T_{1},\ldots ,T_{l}:X\rightarrow X$ be commuting, invertible measure preserving transformations, $f_{1},\ldots ,f_{l}\in L^{\infty }(\mu )$ and $p_{1},\ldots ,p_{l}\in \mathbb{Z}[t]$, such that the \emph{nonconventional averages} involves considering  the multiple ergodic average 
\begin{equation}
\frac{1}{N}S_{N}(x)=\frac{1}{N}\sum_{i=1}^{N}f_{1}(T_{1}^{p_{1}(i)}x)\cdots
f_{l}(T_{l}^{p_{l}(i)}x)\text{,}  \label{3}
\end{equation}%
or more broadly, 
\begin{equation}
\frac{1}{N}S_{N}(x)=\frac{1}{N}\sum_{i=1}^{N}\Phi (T_{1}^{p_{1}(i)}x,\ldots
,T_{l}^{p_{l}(i)}x),  \label{4}
\end{equation}%
where $\Phi :X^{l}\rightarrow \mathbb{R}$. It is important to note that
characterizing the convergence (norm, almost everywhere) of (\ref{3}) or (\ref%
{4}) is extremely challenging (cf.\cite{assani2010pointwise,
bourgain1990double,  frantzikinakis2011some, host2005nonconventional,
krause2022pointwise}).

 Kifer \cite{kifer2010nonconventional, kifer2013strong} and Kifer-Varadhan 
\cite{kifer2014nonconventional} initially
considered the LDP problem for the multiple ergodic averages of the form 
\[
\frac{1}{N}\sum_{i=1}^{N}F \left( X(q_{1}(i)), \ldots, X(q_{l}(i)) \right)\text{,} 
\]%
where $X(i)$, $i\geq 0$ is a Markov process satisfying Doeblin's condition,
and $F$ is a locally H\"{o}lder continuous function with polynomial growth.
Later, Carinci et al. \cite{carinci2012nonconventional} addressed the LDP problem for the multiple average along arithmetic progressions below
on $\mathbb{N}$.
\[
\frac{1}{N}S_{N}(\{X_{i}\})=\frac{1}{N}%
\sum_{i=1}^{N}f_{1}(X_{i})f_{2}(X_{2i}) \cdots f_{l}(X_{li})
\]
and especially on the average, is of the form (\ref{5}).  In \cite%
{carinci2012nonconventional}, the \emph{rate function} associated with the
multiple ergodic average $S_{N}/N$ (\ref{5}) is defined by 
\begin{equation}
I_{r}(x)=\lim_{\epsilon \rightarrow 0}\lim_{N\rightarrow \infty }-\frac{1}{N}%
\log \mathbb{P}_{r}\left( \frac{S_{N}}{N}\rightarrow \lbrack x-\epsilon
,x+\epsilon ]\right) \text{.}  \label{1}
\end{equation}%
The authors prove that (\ref{1}) exists and satisfies the \emph{Fenchel-Legendre transform} of the free energy function as follows:
\[
I_{r}(x)=\sup_{\beta \in \mathbb{R}} \left\{ \beta x-F_{r}(\beta ) \right\} \text{,} 
\]
where $F_{r}(\beta )$ is the \emph{free energy function} defined in (\ref{6}%
). If $F_{r}(\beta )$ is differentiable, say $F_{r}^{\prime }(\eta )=y$, the
rate function can be clearly demonstrated as $I_{r}(y)=\eta y-F_{r}(\eta )$.
Therefore, the explicit formula for the free energy function is significant
for the rate function of the LDP problem.

Let $\mathcal{A}=\{+1,-1\}$ and $\mathbb{P}_{p}$ be the product of Bernoulli
with parameter $p$ on $\mathcal{A}$. For $\sigma \in \mathcal{A}^{\mathbb{Z}%
} $, the authors \cite{chazottes2014thermodynamic} investigate the
thermodynamic limit of the \emph{free energy function} associated with the
multiple ergodic average 
\begin{equation}
\frac{1}{N}S_{N}(\sigma )=\frac{1}{N}\sum_{i=1}^{N}\sigma _{i}\sigma _{2i},
\label{5}
\end{equation}%
defined as 
\begin{equation}
F_{r}(\beta )=\lim_{N\rightarrow \infty }\frac{1}{N}\log \mathbb{E}%
_{r}\left( e^{\beta S_{N}}\right) \text{.}  \label{6}
\end{equation}%
If we consider the multiple ergodic sum $S_{N}$ (\ref{5}) as a Hamiltonian
and the parameter $\beta $ as the `inverse temperature' in the lattice spin
system $\mathcal{A}^{\mathbb{Z}}$, this is the simplest version of the \emph{%
multiplicative Ising model} defined in \cite{chazottes2014thermodynamic}. In contrast to the Hamiltonian of the standard Ising model\footnote{The Hamiltonian is of the form $S_{N}(\sigma )=\sum_{i=1}^{N} \sigma_{i}\sigma _{i+1}$.}, such an interaction is much harder to process because it fails to be translation-invariant and is a long-range interaction (cf. 
\cite{carinci2012nonconventional, chazottes2014thermodynamic}). The
thermodynamic limit, entropy and existence of the Gibbs measure with respect
to multiple ergodic averages $S_{N}(\sigma )/N$ are shown in \cite{chazottes2014thermodynamic}.

A natural question arises: can we extend the above-mentioned LDP result to a multidimensional multiplicative Ising model? The question is interesting since multidimensional Ising models reveal rich and different phenomena on the thermodynamic limit and phase transition problems, compared to one-dimensional ones (cf.\cite{dembo2009ldp, georgii2011gibbs}). Recently, the LDP and explicit formula of the free energy function for the multiple ergodic average have been extended to $\mathbb{N}^{d}$, $d\geq 2$ \cite{ban2021large}, and the main purpose of this paper is to establish the LDP for a multiplicative Ising model on a general tree $\mathcal{T}$. We stress that although $\mathbb{N}^{d}$ and $%
\mathcal{T}$ are both multidimensional lattices, the methods of establishing the LDP for those lattices are quite different. The reason is that the concept of `amenability' is not true anymore for $\mathcal{T}$, that is, $\left\vert \mathcal{T}_{n}\right\vert /\left\vert \Delta _{n}\right\vert $ does not tend to $0$ as $n\rightarrow \infty $, where $\Delta _{n}$ (resp. $\mathcal{T%
}_{n}$) is the set of all vertices of $\mathcal{T}$ with distance from the
root being at most $n$ (resp. $=n$), and $\left\vert A\right\vert $ stands for the cardinality of the set $A$.

Let $\mathcal{T}$ be a \emph{Markov-Cayley tree}. Precisely, if we write $\mathcal{T}$ as a countable graph with a root $\epsilon $ and without loops, there exists an alphabet $\Sigma =\{1,\ldots ,d\}$, such that the set of the vertices of $\mathcal{T}_{n}\subseteq \Sigma ^{n}$  can be identified as the `admissible $n$-words' with respect to an \emph{adjacency matrix} $M$, that is, 
\[
\mathcal{T}_{n}= \left\{(x_{1},\ldots ,x_{n})\in \Sigma ^{n}:M(x_{i},x_{i+1})=1%
\text{ for }i=1,\ldots ,n-1 \right\}\text{.} 
\]
For instance, the adjacency matrix for the conventional $d$-tree is a full $%
d\times d$ matrix over $\{0,1\}$. In this paper, we focus on the broad class of trees, namely, the Markov-Cayley trees that satisfy the following property: 
\begin{equation}
\lim_{n\rightarrow \infty }\frac{\left\vert \mathcal{T}_{n+1}\right\vert }{%
\left\vert \mathcal{T}_{n}\right\vert }=\gamma >1\text{.}  \label{2}
\end{equation}%
This indicates that the vertices of the $n$-levels of $\mathcal{T}$ increase exponentially.  When $M$ is a square matrix, we denote by $\left\Vert M \right\Vert $ the matrix norm of $M$, i.e., $\left\Vert M\right\Vert = \mathbf{1}^{t}M\mathbf{1},$ $\mathbf{1}=(1,1,\ldots ,1)^t$. Since $\left\vert  \mathcal{T}_{n}\right\vert =\left\Vert M^{n-1}\right\Vert$, $n\geq 2$ for a Markov-Cayley tree $\mathcal{T}$, the condition (\ref{2}) is equivalent to 
\begin{equation}
\lim_{n\rightarrow \infty }\frac{\left\Vert M^{n+1}\right\Vert }{\left\Vert
M^{n}\right\Vert }=\gamma >1\text{.}    \label{add22}
\end{equation}%
Hence, we assume that $\mathcal{T}$ is a Markov-Cayley tree that satisfies (%
\ref{add22}) in what follows. It should be pointed out that condition (\ref{2})
is a general condition if T is not a Markov-Cayley tree. We emphasize that
the dynamics of the shift on a tree $\mathcal{T}$ (called \emph{tree-shift})
are quite interesting and have attracted a lot of attention
recently (cf.\cite{ban2017mixing, ban2021characterization, ban2021structure, 
ban2020complexity, ban2021entropy, ban2023entropy, PS-2017complexity,  petersen2020entropy}).

In this article, we consider a larger category of  the (nonconventional)
ergodic sums as well, namely, $S_{N}^{\bigstar }=\sum_{i=1}^{N}\sigma
_{i}\sigma _{ia(i)}$, where $\{a(i)\}_{i\geq 1}\subseteq \mathbb{N}$ with $%
a(i)\geq 2$ $\forall i\geq 1$. Let $\mathbb{P}_{p}$ be the product of
Bernoulli with the parameter $p$ on $\mathcal{A}=\{+1,-1\}$, we define the
associated \emph{free energy function} as follows:
\begin{equation}
F_{p}(\beta )=\limsup\limits_{N\rightarrow \infty }\frac{\log \mathbb{E}%
_{p}(e^{\beta S_{N}^{\bigstar }})}{\left\vert \Delta _{Na(N)-1}\right\vert }\text{,}  \label{one11}
\end{equation}
and the \emph{large deviation rate function} of the multiple average $%
S_{N}^{\bigstar }/\left\vert \Delta _{Na(N)-1}\right\vert $ on tree $\Delta _{Na(N)-1}$
is defined as 
\begin{equation}
I_{p}(x)=\lim_{\epsilon \rightarrow 0}\lim_{N\rightarrow \infty }-\frac{1}{%
\left\vert \Delta _{Na(N)-1}\right\vert }\log \mathbb{P}_{p}\left( \frac{%
S_{N}^{\bigstar }}{\left\vert \Delta _{Na(N)-1}\right\vert }\in \lbrack x-\epsilon
,x+\epsilon ]\right) \text{.}   \label{one12}
\end{equation}
In view of $S_{N}^{\bigstar }$, we deal with the following two types:

\textbf{I}. Let $a(i)=i$ for all $i\geq 1$. In this circumstance, the
ergodic sum is of the form 
\[S_{N}^{\#}=\sum_{i=1}^{N}\sigma _{i}\sigma
_{i^{2}} ,
\]
and here we remark that $a(1)=1$ is allowed and let $\sigma
_{1}\sigma _{1}=1$. The corresponding multiple average is $%
S_{N}^{\#}/\left\vert \Delta _{N^{2}-1}\right\vert $.

\textbf{II}. Let $a(i)=q$\textbf{\ }$\forall i\geq 1$ and $2\leq q\in 
\mathbb{N}$, the ergodic sum  is of the form 
\[S_{N}^{(q)}=\sum_{i=1}^{N}%
\sigma _{i}\sigma _{qi}  
\]
with the average $S_{N}^{(q)}/\left\vert \Delta_{qN-1}\right\vert $.

\vbox{}

For Type I, we have the following results.

\begin{thm}[LDP for Type I]
\label{Thm: 1}The following assertions hold true.

\begin{enumerate}
\item The explicit formula of the free energy function corresponding to the sum $%
S_{N}^{\#} $ is 
\[
F_{p}(\beta )=\frac{\gamma -1}{\gamma }\log \left[ pe^{\beta
}+(1-p)e^{-\beta }\right] +\mathcal{G}_{\#}(\beta )\text{,} 
\]%
where $\mathcal{G}_{\#}(\beta )$ is defined in \eqref{ssa2}.

\item The function $\beta \rightarrow F_{p}(\beta )$ is differentiable.

\item The multiple average $S_{N}^{\#}/\left\vert \Delta
_{N^{2}-1}\right\vert $ satisfies a large deviation principle with the rate
function 
\[
I_{p}(x)=\sup_{\beta \in \mathbb{R}}(\beta x-F_{p}(\beta ))\text{.} 
\]%
Moreover, if $\left( F_{p}\right) ^{\prime }(\eta )=y$, then $I_{p}(y)=\eta
y-F_{p}(\eta )$.
\end{enumerate}
\end{thm}

We note that if $p=\frac{1}{2}$, it can be easily computed that $\mathcal{G}%
_{\#}(\beta )=0$ and 
\begin{equation}
F_{\frac{1}{2}}(\beta )=\frac{\gamma -1}{\gamma }\log \left[ \frac{1}{2}\left(
e^{\beta }+e^{-\beta }\right) \right] .  \label{9}
\end{equation}%
Meanwhile, the rate function can be computed rigorously as 
\[
I_{\frac{1}{2}}(y)=\frac{\gamma -1}{\gamma }\left( \eta \frac{e^{\eta
}-e^{-\eta }}{e^{\eta }+e^{-\eta }}-\log \frac{e^{\eta }+e^{-\eta }}{2}%
\right) \  \text{with}  \  \  y =  \frac{\gamma -1}{\gamma } \cdot \frac{e^{\eta
}-e^{-\eta }}{e^{\eta }+e^{-\eta }} .
\]%

In particular, if $a(i)=i^{\alpha -1}$ for any $i\geq 1$ and a given integer $\alpha \geq 2$, the sum $S_{N}^{\#}(\alpha )=\sum_{i=1}^{N}\sigma
_{i}\sigma _{i^{\alpha }}$ is the generalized case of $S_{N}^{\#}$, and Corollary \ref{co12} below demonstrates that the free energy function $F_{p}(\beta )$ is a constant function for all $\alpha \geq 2$ and $p=1/2$.

\begin{coro}\label{co12}
For $p=1/2$, the free energy function of $F_{p}(\beta )$ with respect to the
ergodic sum $S_{N}^{\#}(\alpha )$ equals (\ref{9}) for all $\alpha \geq 2$.
\end{coro}

For Type II, we have the following results.

\begin{thm}[LDP for Type II]   \label{Thm: 2}
The following assertions hold true.

\begin{enumerate}
\item The free energy function corresponding to the sum $S_{N}^{(q)}$ is 
\[
F_{p}(\beta )=\frac{\gamma ^{q-1}\left( \gamma -1\right) }{\gamma ^{q}-1}%
\log \left[ pe^{\beta }+(1-p)e^{-\beta }\right] +\mathcal{G}_{q}(\beta )%
\text{,} 
\]%
where $\mathcal{G}_{q}(\beta )$ is given by \eqref{ssa1}.

\item The function $\beta \rightarrow F_{p}(\beta )$ is differentiable.

\item The multiple average $S_{N}^{(q)}/\left\vert \Delta _{qN-1}\right\vert 
$ defined on $\mathcal{T}$ satisfies a large deviation principle with the
rate function 
\[
I_{p}(x)=\sup_{\beta \in \mathbb{R}}(\beta x-F_{p}(\beta ))\text{.} 
\]%
Moreover, if $\left( F_{p}\right) ^{\prime }(\eta )=y$, then $I_{p}(y)=\eta
y-F_{p}(\eta )$.
\end{enumerate}
\end{thm}

We remark that if $p=\frac{1}{2}$, it can be easily seen that $\mathcal{G}%
_{q}(\beta )=0$ and 
\[
F_{\frac{1}{2}}(\beta )=\frac{\gamma ^{q-1}\left( \gamma -1\right) }{\gamma
^{q}-1}\log \left[ \frac{1}{2}\left( e^{\beta }+e^{-\beta }\right) \right] 
\text{.}
\]%
Thus, the rate function is 
\[
I_{\frac{1}{2}}(y )=\frac{\gamma ^{q-1}\left( \gamma -1\right) }{\gamma
^{q}-1}\left( \eta \frac{e^{\eta }-e^{-\eta }}{e^{\eta }+e^{-\eta }}-\log 
\frac{e^{\eta }+e^{-\eta }}{2}\right) \text{,}
\]%
where 
\[
y=\frac{\gamma ^{q-1}\left( \gamma -1\right) }{\gamma ^{q}-1}\left( \frac{%
e^{\eta }-e^{-\eta }}{e^{\eta }+e^{-\eta }}\right) \text{.}
\]

\vbox{}

For $q=2$, the Ising sum $S_{N}^{(2)}$ is considered as follows when $\mathcal{T}$ is a $d$-tree or the golden-mean tree (defined in Section 2). The formula for $F_{p}(\beta )$ and $\mathcal{G}_{2}(\beta )$ can
be computed explicitly.

\begin{ex}\label{core12}
The following assertions hold true.

\noindent \textbf{1.} For $d \geq 2$, let $\mathcal{T}$ be a conventional $d$-tree, then the matrix $M$ is the full matrix. This means that $R(k+1,j) = d^k$ and $C(k-1,j) = d^{k-2}$  (which are defined in Section 2 and represent the row sum and column sum, respectively)  for all $k \geq 1$ and $1 \leq j \leq d$. From Theorem \ref{Thm: 2}, the free energy function associated with the sum $S_{N}^{(2)}$ is given by
\[
F_p (\beta) = \frac{d}{d+1} \log \left[pe^{\beta} + (1-p)e^{-\beta}\right] + \mathcal{G}_{2}(\beta)
\]
with 
\[
\mathcal{G}_{2}(\beta) = \limsup_{N \to \infty} \frac{d-1}{d}  \!\sum_{k= \lfloor \!N/2 \rfloor \!+ \!1}^{N} \frac{1}{d^{2N\!-k}} \log \left[ 1 + \frac{1\!-\!p}{p} \left( \frac{(1\!-\!p)e^{\beta} + pe^{-\beta}}{pe^{\beta} + (1\!-\!p)e^{-\beta}}  \right)^{d^k} \right] .
\]

\noindent \textbf{2.}  Suppose that $\mathcal{T}$ is the golden-mean tree,  then it satisfies $\gamma = \frac{\sqrt{5} + 1}{2}$ and the associated transition matrix $M$ whose rows are $11$, $10$. Then we have $R(k\!+\!1,j) = (M^k)_{j,1} + (M^k)_{j,2}$ and $C(k\!-\!1,j) = (M^{k-2})_{1,j} + (M^{k-2})_{2,j}$ for all $k \geq 1$ and $j = 1,2$.  According to Theorem \ref{Thm: 2}, the free energy function related to the sum  $S_{N}^{(2)}$ is obtained as
\[
F_p (\beta) = \frac{\sqrt{5} - 1}{2} \log \left[pe^{\beta} + (1-p)e^{-\beta}\right] + \mathcal{G}_{2}(\beta) ,
\]
where
\[
\mathcal{G}_{2} (\beta)  = \limsup_{N \to \infty}  \sum_{k= \lfloor \!N/2 \rfloor + 1}^{N} \sum_{j=1}^{2} \frac{C(k\!-\!1,j)}{|\Delta_{2N-1}|}  \log \left(1 + \frac{E_{j}^{-}}{E_{j}^{+}} \right) .
\]
\end{ex}

\vbox{}

These results are far from conclusive, and we end this section by pointing out the following problem.

\begin{prob}
Is it possible to address the LDP problem for a multiplicative Ising model on a Markov-Cayley tree $\mathcal{T}$ with respect to the multiple ergodic average 
\[
\frac{1}{N}S_{N}(\sigma )=\frac{1}{N}\sum_{n=1}^{N}\sigma _{a(n)}\sigma
_{b(n)},
\]
where $a(n)=pn+c$ and $b(n)=qn+d$, where $\gcd(p,q)=1$ and $c,b\in \mathbb{Z}$?
\end{prob}

\vbox{}

\section{Preliminaries}

  Let $\mathcal{T}$ be the tree, as it can be also viewed as the set of all finite words on a finite alphabet $\Sigma$. Every word $u$ on $\Sigma$ corresponds to  a vertex of $\mathcal{T}$, with the empty word $\epsilon$ corresponding to the root. Given $\nu \in \mathcal{T}$, we define the length of the unique path from $\epsilon$ to $\nu$ as $|\nu|$.  Then, for each $n \geq 0$ the set of vertices of the $n$-th level of $\mathcal{T}$ is denoted by $\mathcal{T}_n = \left\{ \nu \in \mathcal{T} : |\nu|=n \right\}$ and the $n$-th subtree of $\mathcal{T}$ is defined by $\Delta_n = \bigcup_{i=0}^{n} \mathcal{T}_i$, which implies all vertices in $\Delta_n$ whose distance from $\epsilon$ is at most $n$. Given a vertex $\nu \in \mathcal{T}$,  denote by $e(\nu)$ the number of edges of $\nu$. For $d \geq 2$,  a tree $\mathcal{T}$ is called the $d$-tree if  $e(\epsilon) = d$ and $e(\nu) = d+1$ for all $\nu \neq \epsilon$. Then it can be checked that  $|\Delta_n| = 1 + d + \cdots + d^n$ for the $d$-tree.

Let $\Sigma = \{1,2,\cdots,d\}$ and $M = [m_{i,j}]$ be a $d \times d$ essential 0-1 transition matrix.  A matrix is \emph{essential} if it has no zero rows and zero columns.  Denote by $\Sigma^{*}_{M}$ the collection of all finite words generated by $M$ together with the empty word $\epsilon$. Then the Markov-Cayley tree is defined as $\mathcal{T} = \Sigma^{*}_{M}$, which implies that 
each node of $\mathcal{T}$ is a finite word in $\Sigma^{*}_{M}$. And we have $|\mathcal{T}_1| = d$ and $|\mathcal{T}_n| = \|M^{n-1}\|$ for $n \geq 2$. It is clear to see that
\[
|\Delta_n| = 1 + d + \sum_{j=2}^{n} \left\Vert M^{j-1} \right\Vert .
\]

For example, if we consider $\Sigma = \{1,2\}$ and the matrix $M = \begin{bmatrix}
1 & 1  \\  1 & 0  \end{bmatrix}$, the collection of finite words corresponding to $M$ with the empty word $\epsilon$ can be obtained as $\Sigma^{*}_{M} = \{ \epsilon,1,2,11,12,21,111,112,121,211,212,\cdots \}$, as shown in Figure \ref{fig1}.  Clearly, the golden-mean tree is also an expandable tree with $\gamma = \frac{\sqrt{5} + 1}{2}$.  However, there are trees that are not Markov-Cayley, but they meet the condition \eqref{2},  e.g. the exploding tree in  \cite{benjamini1994markov}.

\begin{figure}[htbp]
\includegraphics[scale=0.65]{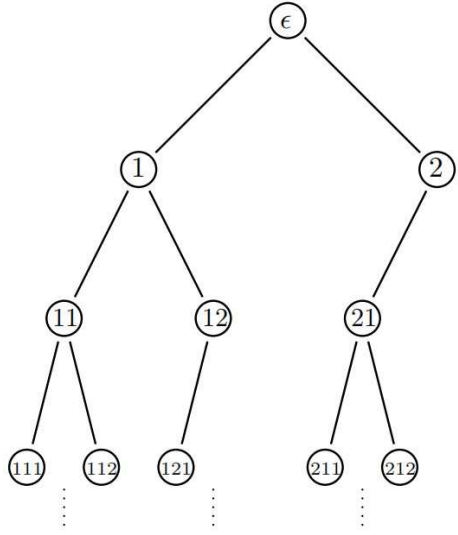}
\caption{Golden-mean tree.}   \label{fig1}
\end{figure}

Let $\mathcal{T} = \Sigma_{M}^*$, where $M$ is a $d \times d$ essential binary matrix.  For $k \geq 2$ and $1 \leq i \leq d$, the number of nodes (i.e. finite words) beginning with `$i$'
on the $k$-th level of $\mathcal{T}$ is equal to the \emph{row sum} 
\[
R(k,i) = \sum_{j = 1}^{d} (M^{k-1})_{i,j}
\]
of the matrix $M^{k-1}$. Similarly, for $1 \leq j \leq d$ the number of nodes (i.e. finite words) ending with `$j$' on the $k$-th level of $\mathcal{T}$ is equal to the \emph{column sum} 
\[
C(k,j) = \sum_{i = 1}^{d} (M^{k-1})_{i,j}
\]
of the matrix $M^{k-1}$.
Notably, the row sum $R(k,i)$ is essentially the number of nodes (i.e. finite words which have length $k$) formed by the symbol `$i$' through $k-1$ steps from the first level  to the $k$-th level on $\mathcal{T}$. In other words, $R(k,i)$ is also the number of paths formed by this symbol `$i$' from the $1$-st level  to the $k$-th level on $\mathcal{T}$.

\vbox{}

In the following, we consider spin systems with Ising $\pm 1$ spins on a Markov-Cayley tree $\mathcal{T}$.  Let $\mathbb{P}_p$ be the product of Bernoulli with the parameter $p$ on two symbols $\mathcal{A} = \{+1,-1\}$. A \emph{configuration} $\sigma$ is an element of the space $\mathcal{A}^{\mathcal{T}}$. We study the \emph{ergodic sum}  $S_{N}^{\bigstar} = \sum_{i=1}^{N} \sigma_i \sigma_{i a(i)}$
defined on $\mathcal{T}$ as the ``Hamiltonian", where $\{ a(i) \}_{i=1}^{\infty}$ is an integer sequence with $a(i) \geq 2$ and $\sigma_{i} \in \mathcal{T}_{i -1}$ for all $i \geq 1$. 
The following theorem shows the connection between the rate function and the free energy function.

\begin{thm}\label{th01}
\textnormal{(G\"artner-Ellis theorem \cite{dembo2009ldp})} 
If the limit \eqref{one11} exists and such limit is finite in a neighborhood of origin, then the rate function \eqref{one12} is equal to the Fenchel-Legendre transform of \eqref{one11}, namely 
\[
I_p (x) = \sup_{\beta \in \mathbb{R}} \left( \beta x - F_p (\beta)  \right) .
\]
Furthermore, if  $F_p (\beta)$ is a differentiable function, written as $ (F_p)'(\eta) =y$, then we have
\[
I_p (y) = \eta y -  F_p (\eta)  .
\]
\end{thm}

\vbox{}

The rest of this paper is devoted to proofs of Theorems \ref{Thm: 1} and \ref{Thm: 2}.

\vbox{}

\section{Proof of Theorem \ref{Thm: 1}}

In this section, we deal with the LDP problem of Type I on a Markov-Cayley tree $\mathcal{T}$. The multiplicative Ising model $S_{N}^{\#} = \sum_{i=1}^{N} \sigma_i \sigma_{i^2}$ defined  on $\mathcal{T}$ will be considered. Denote by  $\mathcal{D}_{\text{I}} = \{-1,1\}^{\Delta_{N^2 -1}}$ the subspace. Then, the  partition function of $S_{N}^{\#}$ is represented by
\[
Z(\beta) = \sum_{\sigma \in \mathcal{D}_{\text{I}}} \exp\!\left\{ \beta \sum_{i=1}^{N} \sigma_i \sigma_{i^2} \right\}
\]
with coupling strength $\beta$ and free boundary conditions.  
Proposition \ref{two11} below demonstrates how the free energy function is affected by the tree structure.

\begin{prop}\label{two11}
The free energy function of the sum $S_{N}^{\#}$  defined on a Markov-Cayley tree that satisfies \eqref{add22}  is rewritten as
\[
\mathcal{F}_p (\beta) = \limsup_{N \to \infty} \frac{1}{|\Delta_{N^2 -1}|} \log \mathbb{E}_p \left(\exp \!\left\{ \beta \!\sum_{i=\lfloor \sqrt{N} \rfloor \!+ \!1}^{N} \!\sigma_i \sigma_{i^2} \right\} \right)  .  
\]
\end{prop}

\begin{proof} 
Without loss of generality, we assume $\beta > 0$. Divide the sum $S_{N}^{\#}$ into the following two parts
\[
S_{N,1}^{\#} = \sum_{i=1}^{\lfloor \sqrt{N} \rfloor } \sigma_i \sigma_{i^2} \  \   \text{and}  \  \   S_{N,2}^{\#} = \sum_{i=\lfloor \sqrt{N} \rfloor + 1}^{N} \sigma_i \sigma_{i^2}.
\]
Then we have
\[
Z(\beta) = \sum_{\sigma \in \mathcal{D}_{\text{I}} } \exp\!\left\{ \beta(S_{N,1}^{\#} + S_{N,2}^{\#} ) \right\}  = \sum_{\sigma \in \mathcal{D}_{\text{I}}} \exp\!\left\{ \beta S_{N,1}^{\#} \right\} \cdot \exp\!\left\{ \beta S_{N,2}^{\#} \right\}  .
\]
Based on the decomposition above, one labels every node of the subtree $\Delta_{ \lfloor \sqrt{N} \rfloor^2 -1}$ with symbol  `$+1$'. Thus, the upper bound of $S_{N,1}^{\#}$ can be obtained as
\begin{equation}
\max \!\left(S_{N,1}^{\#} \right) = \sum_{i=1}^{\lfloor \sqrt{N} \rfloor }  |\mathcal{T}_{i^2 -1} |  =   1 +  \sum_{i=2}^{\lfloor \sqrt{N} \rfloor } \|M^{i^2 -2}\|  \label{ash1}
\end{equation}
(this is not the only way to reach the upper bound).
Since the symbols `$+1,-1$' are arbitrary, it follows $- \max(S_{N,1}^{\#}) \leq S_{N,1}^{\#} \leq \max(S_{N,1}^{\#} )$. This implies
\begin{equation}
\sum_{\sigma \in \mathcal{D}_{\text{I}}} \frac{1}{M(\beta)} \exp\!\left\{ \beta S_{N,2}^{\#} \right\}  \leq Z(\beta) \leq \sum_{\sigma \in \mathcal{D}_{\text{I}}} M(\beta) \exp\!\left\{ \beta S_{N,2}^{\#} \right\} ,
\end{equation}
where $M(\beta) = \exp\!\left\{ \beta \max (S_{N,1}^{\#})  \right\}$.  On the other hand, there exist non-negative integers $a_1 , a_2 , b_1 , b_2$ satisfying $a_1 + b_1 = |\Delta_{\lfloor \sqrt{N} \rfloor^2 -1}|$ and $a_2 + b_2 = |\Delta_{\lfloor \sqrt{N} \rfloor^2 -1}|$ such that
\begin{equation}
p^{a_1} (1-p)^{b_1} \frac{1}{M(\beta)} \leq \mathbb{E}_p \left( \exp \{ \beta S_{N,1}^{\#} \} \right) \leq p^{a_2} (1-p)^{b_2} M(\beta) .
\end{equation}
 Using \eqref{add22} and \eqref{ash1}, we easily see that
\begin{equation}
 \limsup_{N \to \infty} \frac{1}{|\Delta_{N^2 -1}|} \log \mathbb{E}_p \left(\exp \{ \beta S_{N,1}^{\#} \}  \right) = 0 .  \label{ash2}
\end{equation}
Hence, combining \eqref{one11} and \eqref{ash2}, one gets
\[
\mathcal{F}_p (\beta) = \limsup_{N \to \infty} \frac{\log \mathbb{E}_p \left( \exp \{ \beta S_{N,2}^{\#} \}  \right)}{|\Delta_{N^2 -1}|}  .
\]
The proof is complete.
\end{proof}

The next two lemmas are crucial to the proof of Theorem \ref{Thm: 1}.

\begin{lem}\label{two12}
For $k \geq 3$ and $1 \leq j \leq d$, the number of nodes (i.e. finite words) ending with `$j$' on the  corresponding $(k-1)$-th level of a Markov-Cayley tree  is
\begin{equation}
C(k-1,j) = \sum_{\ell = 1}^{d} (M^{k-2})_{\ell,j} .
\end{equation}
\end{lem}

\begin{proof}
The result is clear from the definition of the column sum,  so we omit its proof.
\end{proof}

\begin{lem}\label{two13}
For $k \geq 3$ and $1 \leq j \leq d$, under the action of the Ising sum $S_{N}^{\#}$, each node (i.e. finite word) ending with `$j$' on the  corresponding $(k-1)$-th level of a Markov-Cayley tree generates Ising potentials with
\begin{equation}
R \left(k^2 -k+1,j \right) = \sum_{m = 1}^{d} (M^{(k-1)k})_{j,m} 
\end{equation}
vertices on the $(k^2 -1)$-th level of the Markov-Cayley tree.  
\end{lem}

\begin{proof} A notable fact is that if two spins can generate an Ising potential on trees,  they are on the same path. Thus, the Ising model of Type I implies that for $k \geq 3$ and $1 \leq j \leq d$, we only need to find out  how many paths a node ending with `$j$' on the $(k-1)$-th level will generate on the $(k^2 -1)$-th level. 
To accomplish this aim, we view this node ending with `$j$' as a new node (which has length $1$ ) beginning with `$j$'. By this means, we obtain a solution for how many paths the symbol `$j$' on the $1$-th level will generate on the $(k^2 -k+1)$-th level. By the significance of the row sum on Markov-Cayley trees, the symbol `$j$' on the $1$-th level will generate $R(k^2 -k+1,j)$ paths on the $(k^2 -k+1)$-th level.  This finishes the proof.
\end{proof}

\begin{rmk}\label{two14}
For $k \geq 3$ and $1 \leq j \leq d$, the number of nodes on the $(k^2 -1)$-th level of a  Markov-Cayley tree can be expressed as
\begin{equation}
 \| M^{k^2 -2} \| = \sum_{j=1}^{d} R \left( k^2 -k+1,j \right) C(k-1,j ) .  \label{e78}
\end{equation}
\end{rmk}

Now we define the notation of tree blocks. For $y > x \geq 0$, $1 \leq j \leq d$ and $n \geq 1$,  if one node ending with `$j$' of the $x$-th level generates a path with each of $n$ nodes of the $y$-th level on a Markov-Cayley tree $\mathcal{T}$, we call this subsystem of $n+1$ nodes a \emph{tree block} $\mathcal{T}_b (x,y,n,j)$. For example, the tree block in Lemma \ref{two13} is written as $\mathcal{T}_b (k\!-\!1,k^2 \!-\!1,R(k^2 \!-\!k\!+\!1,j),j)$.
Then the expectation of such a tree block $\mathcal{T}_b (k\!-\!1,k^2 \!-\!1,R(k^2 \!-\!k\!+\!1,j),j)$ for any $k \geq 3$ and $1 \leq j \leq d$ can be denoted by $E_j = E_{j}^{+} +E_{j}^{-}$, where 
\begin{equation}
\left\{ 
\begin{array}{rl}
E_{j}^{+}  &=  \  \    p \left[pe^{\beta} + (1-p)e^{-\beta} \right]^{R \left( k^2 -k+1,j \right)} \\
E_{j}^{-}  &=  \   \   (1-p) \left[(1-p)e^{\beta} + pe^{-\beta} \right]^{R \left(  k^2 -k+1,j \right)}
\end{array} \right.    \label{e11}
\end{equation}
are the expectations when this node ending with `$j$' takes $+1$ and $-1$, respectively.

\vbox{}

In the following, we give the detailed proof of Theorem \ref{Thm: 1}.

\begin{proof}[Proof of Theorem \ref{Thm: 1}]
Let us  suppose that $\mathcal{T}$ is a Markov-Cayley tree that satisfies \eqref{add22}.

\noindent \textbf{1.}  By Proposition \ref{two11}, our aim is to compute 
the Ising potentials generated by nodes on the $\lfloor \!\sqrt{N} \rfloor$-th level to the $(N\!-\!1)$-th level of  $\Delta_{N^2\!-\!1}$. That is to say, we can omit the front part $\Delta_{\lfloor \!\sqrt{N} \rfloor \!-\!1}$ of the subtree $\Delta_{N^2\!-\!1}$. However, for any $\lfloor \!\sqrt{N} \rfloor \!+\!1 \leq k \leq N$ and $1 \leq j \leq d$, each node ending with `$j$' on the $(k\!-\!1)$-th level generates Ising  potentials  with only $R \left(k^2 \!-\!k\!+\!1,j \right)$ nodes on the $(k^2\!-\!1)$-th level.  In other words, each tree block $\mathcal{T}_b (k\!-\!1,k^2\!-\!1,R(k^2\!-\!k\!+\!1,j),j)$ is independent in such a situation. Hence using Lemma \ref{two12}, we have
\begin{equation}
\mathbb{E}_p \left(\exp \!\left\{ \beta \!\sum_{k=\lfloor \sqrt{N} \rfloor \!+ \!1}^{N} \!\sigma_k \sigma_{k^2} \right\} \right) = \prod_{k= \lfloor \sqrt{N} \rfloor \!+ \!1}^{N} \prod_{j=1}^{d} \left(E_{j}^{+} + E_{j}^{-} \right)^{C(k\!-\!1,j)} . \label{sdf1}
\end{equation}
Furthermore, taking the logarithm, \eqref{sdf1} yields
\[
\log \mathbb{E}_p \left(\exp \!\left\{ \beta \!\sum_{k=\lfloor \!\sqrt{N} \rfloor \!+ \!1}^{N} \!\sigma_k \sigma_{k^2} \right\} \right) = \!\sum_{k= \lfloor \!\sqrt{N} \rfloor \!+ \!1}^{N}  \sum_{j=1}^{d} C(k\!-\!1,j) \log \left[E_{j}^{+} \left(1 \!+ \!\frac{E_{j}^{-}}{E_{j}^{+}} \right) \right] .  
\]
Therefore, combining  the above equation, \eqref{add22} and Remark \ref{two14}, we can get 

\vspace{-0.35cm}    

\begin{align*}
F_p (\beta) &= \limsup_{N \to \infty} \frac{1}{|\Delta_{N^2 -\!1}|} \!\sum_{k= \lfloor \sqrt{N} \rfloor \!+ \!1}^{N}  \sum_{j=1}^{d} C(k\!-\!1,j) \left[ \log E_{j}^{+} +  \log \left(1 + \frac{E_{j}^{-}}{E_{j}^{+}} \right)  \right]  \\ 
&= \limsup_{N \to \infty}  \!\sum_{k= \lfloor \!\sqrt{N} \rfloor \!+ \!1}^{N} \!\sum_{j=1}^{d}  \frac{C(k\!-\!1,j) R\left(k^2\!-\!k\!+\!1,j\right)}{|\Delta_{N^2 -\!1}|}  \log \left[pe^{\beta} \!+ \!(1\!-\!p)e^{-\!\beta} \right] \!+ \!\mathcal{G}_{\#} (\beta)  \\
&= \limsup_{N \to \infty}  \sum_{k= \lfloor \sqrt{N} \rfloor + 1}^{N}  \frac{\|M^{k^2 -2} \|}{|\Delta_{N^2 -1}|} \log \left[pe^{\beta} + (1-p)e^{-\beta} \right]  +\mathcal{G}_{\#} (\beta)  \\
&= \frac{\gamma-1}{\gamma} \log \left[pe^{\beta} + (1-p)e^{-\beta} \right] + \mathcal{G}_{\#} (\beta) ,
\end{align*}
where
\begin{equation}
\mathcal{G}_{\#}(\beta) = \limsup_{N \to \infty}  \sum_{k= \lfloor \sqrt{N} \rfloor + 1}^{N}  \sum_{j=1}^{d}  \frac{C(k\!-\!1,j)}{|\Delta_{N^2 -1}|} \log \left(1 + \frac{E_{j}^{-}}{E_{j}^{+}} \right) .  \label{ssa2}
\end{equation}

\noindent \textbf{2.}  In order to show that $F_p (\beta)$ is differentiable, we need to demonstrate that both $\mathcal{G}_{\#} (\beta)$ and its derivatives converge uniformly. For the first statement, we define 
\begin{equation}
B (\beta) \!=  \!\frac{(1\!-\!p)e^{\beta} \!+ \!pe^{-\beta}}{pe^{\beta} \!+ \!(1\!-\!p)e^{-\beta}}       \  \   \text{and}  \  \   \mathcal{L} (\beta) = \!\lim_{k \to \infty} \!\log \!\left(1 \!+ \!\!\frac{1\!-\!p}{p}  \left[B(\beta) \right]^{R(k^2 \!-\!k\!+\!1,j)}  \right) .   \label{lg23}
\end{equation}
Then we have 

a. If $B (\beta)< 1$, $\mathcal{L}(\beta) = 0$ which implies $\mathcal{G}_{\#} (\beta) = 0$;

b. If $B(\beta) = 1$, $\mathcal{L}(\beta) = \log (1/p)$ which also implies $\mathcal{G}_{\#} (\beta) = 0$;

c. If $B (\beta)> 1$, we have 
\[
\mathcal{L}(\beta) =  \lim_{k \to \infty} \log \left(\frac{1-p}{p}  \left[B(\beta) \right]^{R(k^2 \!-\!k\!+\!1,j)}   \right)
\]
which implies  

\vspace{-0.35cm}    

\begin{align*}
\mathcal{G}_{\#} (\beta)  &= \limsup_{N \to \infty} \sum_{k= \lfloor \sqrt{N} \rfloor + 1}^{N}  \sum_{j=1}^{d}  \frac{C(k\!-\!1,j)}{|\Delta_{N^2 -1}|}  \log \left(\frac{1-p}{p}  \left[B(\beta) \right]^{R(k^2 \!-\!k\!+\!1,j)} \right)  \\
&= \limsup_{N \to \infty}   \sum_{k= \lfloor \sqrt{N} \rfloor + 1}^{N} \sum_{j=1}^{d} \frac{C(k\!-\!1,j) R(k^2 \!-k\!+\!1,j) }{|\Delta_{ N^2-1}|} \log B (\beta) \\
&=  \frac{ \gamma-1}{\gamma} \log  B (\beta)  .
\end{align*}
In conclusion, $\mathcal{G}_{\#} (\beta)$ has the property of uniform convergence.

For the second statement, we only need to show that the series 
\[
\mathcal{S}_{\#} (\beta) = \limsup_{N \to \infty}  \sum_{k= \lfloor \sqrt{N} \rfloor + 1}^{N}  \sum_{j=1}^{d} \frac{C(k\!-\!1,j)}{|\Delta_{N^2 -1}|} \left[ \log \left(1 + \frac{E_{j}^{-}}{E_{j}^{+}} \right) \right]'  
\]
converges uniformly with regards to $\beta \in \mathbb{R}$.  Actually, for $k \geq 3$ and $1 \leq j \leq d$, we have  

\vspace{-0.35cm}    

\begin{align*}
 \left[ \log \left(1 \!+ \!\frac{E_{j}^{-}}{E_{j}^{+}} \right) \right]'  &= \frac{R(k^2 \!-\!k\!+\!1,j) \frac{1-p}{p}  \left[B(\beta)\right]^{R(k^2 \!-\!k\!+\!1,j)\!-\!1}}{1+ \frac{1-p}{p}  \left[B(\beta) \right]^{R(k^2 \!-\!k\!+\!1,j)}} \cdot  \left[ B(\beta) \right]'  \\
&= \frac{R(k^2 \!-\!k\!+\!1,j) \frac{1-p}{p}  \left[B(\beta) \right]^{R(k^2 \!-\!k\!+\!1,j)\!-\!1}}{1+ \frac{1-p}{p}  \left[B(\beta) \right]^{R(k^2 \!-\!k\!+\!1,j)}} \cdot  \frac{2(1-2p)}{(pe^{\beta} \!+ \!(1\!-\!p)e^{-\beta})^2}   \text{.}
\end{align*}
It is not difficult to see that 

\vspace{-0.35cm}    

\begin{align*}
   \left\arrowvert \left[ \log  \left(1 \!+ \!\frac{E_{j}^{-}}{E_{j}^{+}} \right)    \right]' \right\arrowvert 
 &\leq  \left\arrowvert   \frac{R(k^2 \!-\!k\!+\!1,j) \frac{1\!-\!p}{p}  \!\left[B(\beta) \right]^{R(k^2 \!-\!k\!+\!1,j)\!-\!1}}{ \frac{1-p}{p} \left[B(\beta) \right]^{R(k^2 \!-\!k\!+\!1,j)}} \right\arrowvert    \!\!\left\arrowvert   \frac{2(1\!-\!2p)}{(pe^{\beta} \!+ \!(1\!-\!p)e^{-\!\beta})^2}   \right\arrowvert     \\
&= R(k^2 \!-\!k\!+\!1,j) \left\arrowvert  \frac{pe^{\beta} \!+ \!(1\!-\!p)e^{-\beta}}{(1\!-\!p)e^{\beta} \!+ \!pe^{-\beta}}  \right\arrowvert   \!\left\arrowvert   \frac{2(1-2p)}{(pe^{\beta} \!+ \!(1\!-\!p)e^{-\beta})^2}   \right\arrowvert  \\
&=  R(k^2 \!-\!k\!+\!1,j) \left\arrowvert  \frac{2(1-2p)}{p(1\!-\!p)(e^{\beta} \!+ \!e^{-\beta})^2 \!+ \!(1\!-\!2p)^2}  \right\arrowvert 
\end{align*}
Then, for any bounded closed interval $[a, b] \subseteq \mathbb{R}$, there exists a constant $N_{a,b} \geq 0$ such that
\begin{equation}
\left\arrowvert  \frac{2(1-2p)}{p(1-p)(e^{\beta} + e^{-\beta})^2 + (1-2p)^2}  \right\arrowvert  \leq N_{a,b} . \label{lg33}
\end{equation}
By \eqref{e78}, it follows that 

\vspace{-0.35cm}    

\begin{align*}
\mathcal{S}_{\#} (\beta) &\leq \limsup_{N \to \infty}   \sum_{k= \lfloor \sqrt{N} \rfloor + 1}^{N}  \sum_{j=1}^{d}  \frac{C(k\!-\!1,j) R(k^2 \!-\!k\!+\!1,j) }{|\Delta_{N^2-1}|}  N_{a,b}  \\
&= \limsup_{N \to \infty}  \sum_{k= \lfloor \sqrt{N} \rfloor + 1}^{N}  \frac{\|M^{k^2 -2}\|}{|\Delta_{N^2-1}|}  N_{a,b}  \\
&= \frac{\gamma-1}{\gamma}   N_{a,b} .
\end{align*}
On the basis of the Weierstrass M-test, the uniform convergence of $\mathcal{S}_{\#} (\beta)$ is demonstrated. Therefore, $F_p (\beta)$ is differentiable because of the arbitrariness of $a$ and $b$.

\noindent \textbf{3.}  Applying Theorem \ref{th01} to (1) and (2) of Theorem \ref{Thm: 1}, the result follows.
\end{proof}

\vbox{}

Given an integer $\alpha \geq 2$, we consider the sequence $a(i) = i^{\alpha -1}$ for all $i \geq 1$. The sum $S_{N}^{\#} (\alpha)= \sum_{i=1}^{N} \sigma_i \sigma_{i^{\alpha}}$ can be viewed as the general case of Type I.  If we replace $k^2$ with $k^{\alpha}$ in Lemma \ref{two13}, the tree block associated with $S_{N}^{\#} (\alpha)$ can be written as $\mathcal{T}_b (k\!-\!1,k^{\alpha} \!-\!1,R \left(k^{\alpha} \!-\!k \!+\!1,j \right),j)$ for $k \geq 3$ and $1 \leq j \leq d$.
Then the expectation of such a tree block is represented as $\bar{E}_j = \bar{E}_{j}^{+} +\bar{E}_{j}^{-}$, where
\begin{equation}
\left\{ 
\begin{array}{rl}
\bar{E}_{j}^{+}  &=  \  \    p \left[pe^{\beta} + (1-p)e^{-\beta} \right]^{R \left(k^{\alpha} -k +1,j \right)} \\
\bar{E}_{j}^{-}  &=  \   \   (1-p) \left[(1-p)e^{\beta} + pe^{-\beta} \right]^{R \left(k^{\alpha} -k +1,j \right)}
\end{array} \right.    \label{e33}
\end{equation}
are the expectations when this node ending with `$j$'  takes $+1$ and $-1$, respectively.

\vbox{}

Referring to the previous notations and settings, Corollary \ref{co12} is proven in the following.

\begin{proof}[Proof of Corollary \ref{co12}] 
Similar to the discussion of the proof of Theorem \ref{Thm: 1}, one can see that each tree block $\mathcal{T}_b (k\!-\!1,k^{\alpha} \!-\!1,R \left(k^{\alpha} \!-\!k \!+\!1,j \right),j)$ is also independent on $\Delta_{N^{\alpha} - \!1}$  for $\lfloor \!\sqrt[\alpha]{N} \rfloor \!+\!1 \leq k \leq N$ and $1 \leq j \leq d$. By Lemma \ref{two12}, we can get 
\[
 \mathbb{E}_p \left( \exp \!\left\{ \beta \!\sum_{k=\lfloor \!\sqrt[\alpha]{N} \rfloor \!+ \!1}^{N}  \!\sigma_k \sigma_{k^{\alpha} } \right\} \right) = \prod_{k= \lfloor \!\sqrt[\alpha]{N} \rfloor\!+ \!1}^{N} \prod_{j=1}^{d} \left(\bar{E}_{j}^{+} + \bar{E}_{j}^{-} \right)^{C(k-1,j)} .
\]
It can be verified without difficulty from Proposition \ref{two11} that 

\vspace{-0.35cm}    

\begin{align*}
F_p (\beta) &=  \limsup_{N \to \infty} \frac{1}{|\Delta_{N^{\alpha} -\!1}|} \log \mathbb{E}_p \left(\exp \!\left\{ \beta \!\sum_{k=\lfloor \!\sqrt[\alpha]{N} \rfloor \!+ \!1}^{N} \!\sigma_k \sigma_{k^{\alpha}} \right\} \right)  \\
&= \limsup_{N \to \infty}  \sum_{k= \lfloor \!\sqrt[\alpha]{N} \rfloor \!+ \!1}^{N}  \sum_{j=1}^{d} \frac{C(k\!-\!1,j)}{|\Delta_{N^{\alpha} -\!1}|} \left[ \log \bar{E}_{j}^{+} +  \log \left(1 + \frac{\bar{E}_{j}^{-}}{\bar{E}_{j}^{+}} \right)  \right] .
\end{align*}
Now take $p=\frac{1}{2}$ in the above equation. We can easily see that $\log (1 + \bar{E}_{j}^{-}/\bar{E}_{j}^{+} ) = \log 2$. On the other hand, for  $ k \geq 3$ and $1 \leq j \leq d$, referring to Remark \ref{two14} the number of nodes on the $(k^{\alpha} -1)$-th level can be represented as
\begin{equation}
 \|M^{k^{\alpha} -2} \| = \sum_{j=1}^{d} R \left( k^{\alpha} -k +1,j \right) C(k-1,j ) . \label{sd22}
\end{equation}
Combining  \eqref{add22} and \eqref{sd22} , the free energy function is obtained as 

\vspace{-0.35cm}    

\begin{align*}
F_{\frac{1}{2}} (\beta) &= \limsup_{N \to \infty}  \sum_{k= \lfloor \!\!\sqrt[\alpha]{N} \rfloor \!+ \!1}^{N}  \sum_{j=1}^{d} \frac{C(k\!-\!1,j)}{|\Delta_{N^{\alpha} -\!1}|} \left(\log \bar{E}_{j}^{+} +  \log 2  \right)   \\
&= \limsup_{N \to \infty}  \sum_{k= \lfloor \!\!\sqrt[\alpha]{N} \rfloor \!+ \!1}^{N}  \sum_{j=1}^{d} \frac{C(k\!-\!1,j)}{|\Delta_{N^{\alpha} -\!1}|} \!\left[ \log \frac{1}{2} \!+ \!R\left(k^{\alpha} \!\!-\!k\!\!+\!\!1,j\right) \log \!\frac{e^{\beta} \!\!+ \!e^{-\!\beta} }{2}  \right]    \\
&= \limsup_{N \to \infty}  \sum_{k= \lfloor \!\!\sqrt[\alpha]{N} \rfloor \!+ \!1}^{N}  \sum_{j=1}^{d} \frac{\|M^{k^{\alpha} -2} \|}{|\Delta_{N^{\alpha} -\!1}|} \log \left[\frac{1}{2} \left(e^{\beta} + e^{-\beta} \right) \right]   \\
&=  \frac{\gamma-1}{\gamma} \log \left[\frac{1}{2} \left(e^{\beta} + e^{-\beta} \right) \right]
\end{align*}
That is precisely what we wanted to prove.
\end{proof}

We emphasize that the free energy function is independent of the value of $\alpha$ when $p = \frac{1}{2}$. 

\vbox{}

\section{Proof of Theorem \ref{Thm: 2}}

This section presents the proof for Theorem \ref{Thm: 2}.  Our attention focuses on a multiplicative sum of Type II, namely $S_{N}^{(q)} = \sum_{i=1}^{N} \sigma_i \sigma_{q i}$ for $q \geq 2$, which is defined on a Markov-Cayley tree $\mathcal{T}$. Thus, the state space is $\mathcal{D}_{\text{II}} = \{-1,1\}^{\Delta_{q N-1}}$, and the partition function of $S_{N}^{(q)}$ on $\Delta_{q N-1}$ is written as
\begin{equation}
Z(\beta) = \sum_{\sigma \in \mathcal{D}_{\text{II}}} \exp \!\left\{ \beta \sum_{i=1}^{N} \sigma_i \sigma_{q i} \right\}  \label{e56}
\end{equation}
with free boundary conditions.  Proposition \ref{three11} is analogous to Proposition \ref{two11}. We omit the proof since they are almost identical.

\begin{prop}\label{three11}
The free energy function of the sum $S_{N}^{(q)}$ defined on a Markov-Cayley tree that satisfies \eqref{add22}  is rewritten as
\[
\mathcal{F}_p (\beta) = \limsup_{N \to \infty} \frac{1}{|\Delta_{q N-1}|} \log \mathbb{E}_p \left(\exp \!\left\{ \beta \!\sum_{i=\lfloor \frac{N}{q} \rfloor + 1}^{N} \!\sigma_i \sigma_{q i} \right\} \right)  .  
\]
\end{prop}

The following lemma is useful for the proof Theorem \ref{Thm: 2}.

\begin{lem}\label{three12}
For $k \geq 3$ and $1 \leq j \leq d$, under the action of the Ising sum $S_{N}^{(q)}$, each node (i.e. finite word) ending with `$j$' on the  corresponding $(k\!-\!1)$-th level of a Markov-Cayley tree generates Ising potentials with
\[
R \left((q-1)k+1,j \right) = \sum_{m = 1}^{d} (M^{(q-1)k})_{j,m} 
\]
vertices on the $(qk\!-\!1)$-th level of the Markov-Cayley tree.  The corresponding potential is 
\begin{equation}
U( \{k,q k\},\sigma) = \beta \sum_{i=1}^{R((q\!-\!1)k\!+\!1,j)} \sigma_{k} \sigma_{q k,i} ,  \label{add55}
\end{equation}
where $\sigma_{k}, \sigma_{q k,i}$ are i.i.d. random variables taking values in set $\{-1,+1\}$.
\end{lem}

\begin{proof}
 Refer to the proof of Lemma \ref{two13}. For any $k \geq 3$ and $1 \leq j \leq d$, our goal is to discover how many paths a node ending with `$j$' on the $(k\!-\!1)$-th level will generate on the $(q k\!-\!1)$-th level. This is equivalent to the number of paths that the symbol `$j$' on the $1$-st level generates on the $((q\!-\!1)k\!+\!1)$-th level.  By definition of the row sum, the symbol `$j$' on the $1$-st level will generate $R \left((q\!-\!1)k\!+\!1,j \right)$ paths on the $((q\!-\!1)k\!+\!1)$-th level. It is clear that \eqref{add55} follows.
\end{proof}

\begin{rmk}\label{three13}
For $k \geq 3$ and $1 \leq j \leq d$, the number of nodes on the $(q k-1)$-th level of a Markov-Cayley tree can be expressed as
\[
 \|M^{q k-2}\| = \sum_{j=1}^{d} R \left( (q-1)k+1,j \right) C(k-1,j ) .
\]
\end{rmk}

Recall the definition of tree blocks in Section 3.  According to the multiplicative relation of the Ising model $S_{N}^{(q)}$, for  $k \geq 3$ and $1 \leq j \leq d$ one can get the corresponding tree block written as $\mathcal{T}_b (k\!-\!1,q k\!-\!1,R(q k\!-\!k\!+\!1,j),j)$ . Then the expectation of such a tree block is expressed as $F_j = F_{j}^{+} +F_{j}^{-}$, where

\vspace{-0.35cm}    

\begin{equation}
\left\{ 
\begin{array}{rl}
F_{j}^{+}  &=  \  \    p \left[pe^{\beta} + (1-p)e^{-\beta} \right]^{R \left( (q-1)k+1,j \right)} \\
F_{j}^{-}  &=  \   \   (1-p) \left[(1-p)e^{\beta} + pe^{-\beta} \right]^{R \left( (q-1)k+1,j \right)}
\end{array} \right.    \label{e789}
\end{equation}
are the expectations when this node ending with `$j$'  takes $+1$ and $-1$, respectively.

\vbox{}

The proof of Theorem \ref{Thm: 2} is presented.

\begin{proof}[Proof of Theorem \ref{Thm: 2}]
 Initially, suppose that $\mathcal{T}$ is a Markov-Cayley tree and satisfies \eqref{add22}.

\noindent \textbf{1.}  Proposition \ref{three11} says that the energy distribution on $\Delta_{\lfloor N/q \rfloor-1}$ has no effect on the energy distribution of $\Delta_{q N-1}$ as $N$ approaches infinity. This means that we only need to calculate the Ising potentials generated by nodes on the $\lfloor \!N/q \rfloor$-th level to the $(N\!-\!1)$-th level of  $\Delta_{q N\!-\!1}$.  For $\lfloor \!N/q \rfloor \!+\!1 \leq k \leq N$ and $1 \leq j \leq d$, each node ending with `$j$' of the $(k\!-\!1)$-th level will generate  Ising potentials with only $R((q\!-\!1)k\!+\!1,j)$ nodes of the $(q k\!-\!1)$-th level. This implies that every subsystem $\mathcal{T}_b (k\!-\!1,q k\!-\!1,R(q k\!-\!k\!+\!1,j),j)$ is independent of each other.  Then applying Lemma \ref{two12}, we have
\[
\mathbb{E}_p \left(\exp \!\left\{ \beta \!\sum_{k=\lfloor \!N/q \rfloor \!+ \!1}^{N} \!\sigma_k \sigma_{q k} \right\} \right) = \prod_{k=\lfloor \!N/q \rfloor \!+ \!1}^{N} \prod_{j=1}^{d} \left(F_{j}^{+} + F_{j}^{-} \right)^{C(k-1,j)} ,
\]
which yields
\[
\log \mathbb{E}_p  \left(\exp \!\left\{ \beta \!\sum_{k=\lfloor \!N/q \rfloor \!+ \!1}^{N} \!\sigma_k \sigma_{q k} \right\}\right) = \sum_{k=\lfloor \!N/q \rfloor \!+ \!1}^{N} \sum_{j=1}^{d} C(k\!-\!1,j) \log \left[F_{j}^{+} \left(1 \!+ \!\frac{F_{j}^{-}}{F_{j}^{+}} \right) \right] .  
\]
In connection with \eqref{add22} and Remark \ref{three13}, we obtain

\vspace{-0.35cm}    

\begin{align*}
F_p (\beta) &= \limsup_{N \to \infty}  \!\sum_{k= \lfloor \!N/q \rfloor + 1}^{N} \!\sum_{j=1}^{d}  \frac{C(k\!-\!1,j)}{|\Delta_{q N-1}|} \left[ \log F_{j}^{+} +  \log \left(1 + \frac{F_{j}^{-}}{F_{j}^{+}} \right)  \right]  \\ 
&= \limsup_{N \to \infty}  \!\sum_{k= \lfloor \!N/q \rfloor \!+ \!1}^{N}\!\sum_{j=1}^{d}  \frac{C(k\!-\!\!1,j) R \left( (q\!-\!\!1)k\!+\!\!1,j \right)}{|\Delta_{q N-1}|} \log \left[pe^{\beta} \!\!+ \!(1\!-\!p)e^{-\beta} \right] \!+ \!\mathcal{G}_{q} (\beta)  \\
&= \limsup_{N \to \infty}  \sum_{k= \lfloor \!N/q \rfloor + 1}^{N}  \frac{\|M^{q k-2}\|}{|\Delta_{q N-1}|} \log \left[pe^{\beta} + (1-p)e^{-\beta} \right]  +\mathcal{G}_{q} (\beta)  \\
&= \frac{\gamma^{q-1} (\gamma-1)}{\gamma^q -1} \log \left[pe^{\beta} + (1-p)e^{-\beta} \right] + \mathcal{G}_{q} (\beta) ,
\end{align*}
where
\begin{equation}
\mathcal{G}_{q}(\beta) = \limsup_{N \to \infty}  \sum_{k= \lfloor \!N/q \rfloor + 1}^{N} \sum_{j=1}^{d} \frac{C(k\!-\!1,j)}{| \Delta_{q N-1} |}   \log  \left(1 + \frac{F_{j}^{-}}{F_{j}^{+}} \right)  .  \label{ssa1}
\end{equation}

\noindent \textbf{2.} The discussion of the differentiation of $F_p (\beta)$ is similar to the proof of (2) in Theorem \ref{Thm: 1}. Recalling that $B(\beta)$ in \eqref{lg23}, we let 
\[
\bar{\mathcal{L}} (\beta)= \lim_{k \to \infty} \log \left(1 + \frac{1-p}{p}  \left[B(\beta)\right]^{R((q-1)k+1,j)} \right) .
\]
Then, we use the same method to analyze $B(\beta)$ and $\bar{\mathcal{L}} (\beta)$, which shows that $\mathcal{G}_{q} (\beta)$ is uniformly convergent with respect to $\beta \in \mathbb{R}$. Further, we will prove the convergence of the series  
\[
\mathcal{S}_{q} (\beta) = \limsup_{N \to \infty} \sum_{k=\lfloor \!N/q \rfloor \!+ \!1}^{N} \sum_{j=1}^{d}  \frac{C(k-1,j)}{|\Delta_{q N-1}|} \left[ \log \left(1 + \frac{F_{j}^{-}}{F_{j}^{+}} \right) \right]' .
\]
Indeed, for $k \geq 3$ and $1 \leq j \leq d$ we can verify that 

\vspace{-0.35cm}    

\begin{align*}
 \left[ \log \left(1 \!+ \!\frac{F_{j}^{-}}{F_{j}^{+}} \right) \right]'  &= \frac{R((q\!-\!1)k\!+\!1,j) \frac{1-p}{p}  B^{R((q\!-\!1)k\!+\!1,j)\!-\!1}}{1+ \frac{1-p}{p}  B^{R((q-1)k+1,j)}} \cdot  \left[ B(\beta)  \right]'  \\
&= \frac{R((q\!-\!1)k\!+\!1,j) \frac{1-p}{p}  B^{R((q\!-\!1)k\!+\!1,j)\!-\!1}}{1+ \frac{1-p}{p}  B^{R((q-1)k+1,j)}} \cdot  \frac{2(1-2p)}{(pe^{\beta} \!+ \!(1\!-\!p)e^{-\beta})^2}
\end{align*}
which implies that

\vspace{-0.35cm}    

\begin{align*}
   \left\arrowvert \left[ \log  \left(1 \!+ \!\frac{F_{j}^{-}}{F_{j}^{+}} \right)    \right]' \right\arrowvert 
 &\leq  \left\arrowvert   \frac{R((q\!-\!1)k\!+\!1,j) \frac{1\!-\!p}{p}  B^{R((q\!-\!1)k\!+\!1,j)\!-\!1}}{ \frac{1-p}{p} B^{R((q-1)k+1,j)}} \right\arrowvert    \left\arrowvert   \frac{2(1-2p)}{(pe^{\beta} \!+ \!(1\!-\!p)e^{-\!\beta})^2}   \right\arrowvert     \\
&= R((q\!-\!1)k\!+\!1,j) \left\arrowvert  \frac{pe^{\beta} \!+ \!(1\!-\!p)e^{-\beta}}{(1\!-\!p)e^{\beta} \!+ \!pe^{-\beta}}  \right\arrowvert   \left\arrowvert   \frac{2(1-2p)}{(pe^{\beta} \!+ \!(1\!-\!p)e^{-\beta})^2}   \right\arrowvert  \\
&=  R((q\!-\!1)k\!+\!1,j) \left\arrowvert  \frac{2(1-2p)}{p(1\!-\!p)(e^{\beta} \!+ \!e^{-\beta})^2 \!+ \!(1\!-\!2p)^2}  \right\arrowvert 
\end{align*}
Combining \eqref{lg33} and Remark \ref{three13}, we conclude that 

\vspace{-0.35cm}    

\begin{align*}
\mathcal{S}_{q} (\beta) &\leq \limsup_{N \to \infty}  \sum_{k=\lfloor \!N/q \rfloor \!+ \!1}^{N} \sum_{j=1}^{d}  \frac{C(k\!-\!1,j) R((q\!-\!1)k\!+\!1,j)}{|\Delta_{q N-1}|}  N_{a,b}  \\
&= \limsup_{N \to \infty}  \sum_{k=\lfloor \!N/q \rfloor + 1}^{N}  \frac{\|M^{qk-2}\|}{|\Delta_{q N-1}|}  N_{a,b}  \\
&= \frac{\gamma^{q-1} (\gamma-1)}{\gamma^q -1}   N_{a,b} .
\end{align*}
It is therefore a fact for us to realize by the Weierstrass M-test that the sum $\mathcal{S}_{q} (\beta)$ has uniform convergence. This implies that $F_p (\beta)$ is a differentiable function. 

\noindent \textbf{3.}  The result follows from Theorem \ref{th01} and the statements (1), (2) of Theorem \ref{Thm: 2}.
\end{proof}

\vbox{}

\bibliographystyle{amsplain}  

\bibliography{ban}  

\end{document}